\documentclass[11pt,reqno]{amsart}
\usepackage{tikz}
\usepackage{subfig}
\usepackage{epstopdf}
\usepackage{epsfig}
\usepackage{math}

\graphicspath{{./figures/}}
\usepackage{tikz}
\usetikzlibrary{shapes,arrows}
\tikzstyle{decision} = [diamond, draw, fill=blue!20, 
    text width=4.5em, text badly centered, node distance=3cm, inner sep=0pt]
\tikzstyle{block} = [rectangle, draw, fill=blue!20, 
    text width=5em, text centered, rounded corners, minimum height=4em]
\tikzstyle{line} = [draw, -latex']
\tikzstyle{cloud} = [draw, ellipse,fill=red!20, node distance=3cm,
    minimum height=2em]
    
\usetikzlibrary{positioning}
\tikzset{main node/.style={circle,fill=blue!20,draw,minimum size=1cm,inner sep=0pt},  }

\begin{document}
\title[Lagrangian formulation]{Constrained dynamical optimal transport and its Lagrangian formulation}
\author[Li, Osher]{Wuchen Li and Stanley Osher}
\email{wcli@math.ucla.edu}
\email{sjo@math.ucla.edu}
\address{Department of Mathematics, University of California, Los Angeles.}
\newcommand{\vr}{\overrightarrow}
\newcommand{\wt}{\widetilde}
\newcommand{\dd}{\mathcal{\dagger}}
\newcommand{\ts}{\mathsf{T}}
\keywords{Dynamical optimal transport; Mean field games; Information geometry; Machine learning.}
\thanks{The research is supported by AFOSR MURI proposal number 18RT0073.}
\maketitle
\begin{abstract}
We propose dynamical optimal transport (OT) problems constrained in a parameterized probability subset. In applied problems such as deep learning, the probability distribution is often generated by a parameterized mapping function. In this case, we derive a formulation for the constrained dynamical OT.
\end{abstract}
\section{Introduction}
Dynamical optimal transport problems play vital roles in fluid dynamics \cite{SE} and mean field games \cite{MFG}. They provide a type of statistical distance and interesting differential structures in the set of probability densities \cite{vil2008}. The full probability set is often intractable, when the dimension of sample space is large. For this reason, parameterized probability subsets have been widely considered, especially in machine learning problems and information geometry \cite{Amari, IG, IG2, tryphon, modin, GP}. We are interested in studying dynamical OT problems over a parameterized probability subset. 

In this note, we follow a series of work found in \cite{NG2, LiG, LiG2, NG1, LNG2}, and introduce general constrained dynamical OT problems in a parametrized probability subset. As in deep learning \cite{WGAN}, the probability subset is often constructed by a parameterized mapping. In these cases, we demonstrate that the constrained dynamical OT problems exhibit simple variational structures.

We arrange this note as follows: In section \ref{section2}, we briefly review the dynamical OT in both Eulerian and Lagrangian coordinates\footnote{In fluid dynamics, the Eulerian coordinates represent the evolution of probability density function of particles, while the Lagrangian coordinates describe the motion of particles. In learning problems, the Eulerian coordinates naturally connect with the minimization problem in term of probability densities, while Lagrangian coordinate refers to the variational problem formulated in samples, whose analog are particles. ``In learning, we model problems in Eulerian, and compute them in Lagrangian''. In other words, we often write the objective function in term of densities and compute them via samples.}. Using Eulerian coordinates, we propose the constrained dynamical OT over a parameterized probability subset. In section \ref{section3}, we next derive an equivalent Lagrangian formulation for the constrained problem. 
\section{Constrained dynamical OT}\label{section2}
In this section, we briefly review the dynamical OT in a full probability set via both Eulerian and Lagrangian formalisms. Using the Eulerian coordinates, we propose the dynamical OT in a parameterized probability subset. 

For the simplicity of exposition, all our derivations assume smoothness. Consider densities $\rho^0,\rho^1\in \mathcal{P}_+(\Omega)=\Big\{\rho(x)\in C^{\infty}(\Omega)\colon \rho(x)>0,~\int_\Omega\rho(x)dx=1\Big\}$, where $\Omega$ is a $n$-dimensional sample space. Here $\Omega$ can be $\mathbb{R}^n$, or a convex compact region in $\mathbb{R}^n$ with zero flux conditions or periodic boundary conditions. 

Dynamical OT studies a variational problem in density space, known as the Benamou-Brenier formula \cite{BB}. Given a Lagrangian function $L\colon T\Omega \rightarrow [0,+\infty)$ under suitable conditions, consider
 \begin{subequations}\label{BBL}
 \begin{equation}\label{BB}
C(\rho^0,\rho^1):=\inf_{v_t}~\int_0^1 \mathbb{E}L(X_t(\omega), v(t,X_t(\omega))) dt, 
 \end{equation}
 where $\mathbb{E}$ is the expectation operator over realizations $\omega$ in event space and the infimum is taken over all vector fields $v_t=v(t,\cdot)$, such that 
 \begin{equation}\label{BB2}
 \dot X_t(\omega)=v(t, X_t(\omega)),\quad X_0\sim\rho^0,\quad X_1\sim\rho^1.
 \end{equation}
 \end{subequations}
Here $X_i\sim \rho^i$ represents that $X_i(\omega)$ satisfies the probability density $\rho^i(x)$, for $i=0,1$. Equivalently, denote the density of particles $X_t(\omega)$ at space $x$ and time $t$ by $\rho(t,x)$. Then problem \eqref{BBL} refers to a variational problem in density space:
 \begin{subequations}\label{BBE}
 \begin{equation}\label{BB}
C(\rho^0,\rho^1):=\inf_{v_t}~\int_0^1\int_\Omega L(x, v(t,x))\rho(t,x) dx dt, 
 \end{equation}
 where the infimum is taken over all Borel vector fields $v_t=v(t,\cdot)$, such that the density function $\rho(t,x)$ satisfies the continuity equation:  
 \begin{equation}\label{BB2}
 \frac{\partial \rho(t,x)}{\partial t}+\nabla\cdot (\rho(t,x)v(t,x))=0,\quad \rho(0,x)=\rho^0(x),\quad \rho(1,x)=\rho^1(x).
 \end{equation}
 \end{subequations}
 Here $\nabla\cdot$ is the divergence operator in $\Omega$. In the language of fluid dynamics, problem \eqref{BBL} refers to the Lagrangian formalism, while problem \eqref{BBE} is the associated Eulerian formalism; see details in \cite{vil2008}. The Lagrangian formalism focuses on the motion of each individual particles, while the Eulerian formalism describes the global behavior of all particles.   
 Here \eqref{BBL} and \eqref{BBE} are equivalent since they represent the same variational problem using different coordinate systems. 
 In addition, one often considers $L(x,v)=\|v\|^p$, $p\geq 1$, where $\|\cdot\|$ is the Euclidean norm. In this case, the optimal value of the variational problem defines a distance function in the set of probability space. Denote $W_p(\rho^0,\rho^1):=C(\rho^0,\rho^1)^{\frac{1}{p}}$, where $W_p$ is called the $L^p$-Wasserstein distance. 
 
We next study the variational problem \eqref{BBE} constrained on a parameterized probability density set. In other words, consider a parameter space $\Theta\subset \mathbb{R}^d$ with  
\begin{equation*}
\mathcal{P}_{\Theta}=\Big\{\rho(\theta, x)\in C^{\infty}(\Omega)\colon \theta\in \Theta,~\int_\Omega \rho(\theta,x)dx=1,~\rho(\theta,x)>0, ~x\in \Omega\Big\}.
\end{equation*}
Here we assume that $\rho\colon \Theta \rightarrow \mathcal{P}_+(\Omega)$ is an injective mapping\footnote{We abuse the notation of $\rho$. Notice that $\rho(\theta, x)$ is a probability distribution parameterized by $\theta\in \Theta$, while $\rho(x)$ is a probability distribution function in the full probability set.}. We introduce the constrained dynamical OT as follows:
\begin{subequations}\label{cons}
 \begin{equation}
c(\theta_0,\theta_1):=\inf_{v_t}~\int_0^1\int_\Omega L(x, v(t,x))\rho(\theta_t,x) dx dt,
\end{equation}
where $\theta_t=\theta(t)\in \Theta$, $t\in[0,1]$, is a path in parameter space, and the infimum is taken over the Borel vector fields $v_t=v(t,\cdot)$, such that the {\em constrained continuity equation} holds:
\begin{equation}\label{3b}
 \frac{\partial}{\partial t}\rho(\theta_t,x)+\nabla\cdot (\rho(\theta_t,x)v(t,x))=0,\quad\textrm{$\theta_0, \theta_1$ are fixed}.
 \end{equation}
 \end{subequations}
We notice that the infimum of problem \eqref{cons} is taken over density paths lying in the parameterized probability set, i.e. $\rho(\theta_t, \cdot)\in \rho(\Theta)$. Here the changing ratio of density is reduced into a finite dimensional direction, i.e. $ \frac{\partial}{\partial t}\rho(\theta_t,x)=(\nabla_{\theta_t}\rho(\theta_t, x), \frac{d\theta_t}{dt})$, where $(\cdot,\cdot)$ is an inner product in $\mathbb{R}^d$.

 A natural question arises. Variational problem \eqref{BBE}, together with its constrained problem \eqref{cons}, are written in Eulerian coordinates. They evolve unavoidably with the entire probability density functions. For practical reasons, 
 can we find the Lagrangian coordinates for constrained problem \eqref{cons}? In other words, what are analogs of \eqref{BBL} in $\rho(\Theta)$? We next demonstrate an answer to this question. We show that there is an expression for the motion of particles, whose density path moves according to the constrained continuity equation \eqref{3b}.

\section{Lagrangian formulations}\label{section3}
In this section, we show the main result of this note, that is the constrained dynamical OT \eqref{cons} has a simple Lagrangian formulation in Proposition \ref{thm}. 

Consider a parameterized mapping or implicit generative model as follows. Given a input space $Z\subset \mathbb{R}^{n_1}$, $n_1\leq n$, let
$$g_\theta\colon Z\rightarrow \Omega,\quad x=g(\theta, z)\in \Omega,\quad\textrm{for $z\in Z$.}$$
Here $g_\theta$ is a mapping function depending on parameters $\theta\in \Theta$. 
Given realizations $\omega$ in event space, we assume that the random variable $z(\omega)$ satisfies a density function $\mu(z)\in \mathcal{P}_+(Z)$, and denote $x(\omega)=g(\theta, z(\omega))$ satisfying the density function $\rho(\theta,x)$. This means that the map $g_\theta$ pushes forward $\mu(z)$ to $\rho(\theta,x)$, denoted by $\rho(\theta, x)=g_\theta\sharp\mu(z)$: \begin{equation}\label{a}
\int_{Z}f(g(\theta, z))\mu(z)dz=\int_{\Omega}f(x)\rho(\theta,x)dx, \quad\textrm{for any $f\in C_c^{\infty}(\Omega)$.}
\end{equation}
In this case, the parameterized probability set is given as follows:
\begin{equation*}
\rho(\Theta)=\Big\{\rho(\theta, x)\in C^{\infty}(\Omega)\colon \theta\in \Theta,~\rho(\theta, x)=g_\theta\sharp\mu(z)\Big\}.
\end{equation*}

We next present the constrained dynamical OT in Lagrangian coordinates. We notice a fact that, the vector field in optimal density path of problem \eqref{BBE} or \eqref{cons} satisfies 
$$v(t,x)=D_pH(x,\nabla\Phi(t,x)),$$
where $$H(x,p)=\sup_{v\in T_x\Omega} p v- L(x,v)$$ is the Hamiltonian function associated with $L$.
\begin{proposition}[Constrained dynamical OT in Lagrangian formulation]\label{thm}
The constrained dynamical OT has the following formulation:   
\begin{equation}\label{Lag}
\begin{split}
c(\theta_0, \theta_1)=\inf \Big\{&\int_0^1\mathbb{E}_{z\sim\mu}L(g(\theta_t,z), \frac{d}{dt}g(\theta_t, z))dt\colon\\
& \frac{d}{dt}g(\theta_t, z)=D_p H(g(\theta_t,z), \nabla_x\Phi(t, g(\theta_t,z))),~\theta(0)=\theta_0,~\theta(1)=\theta_1\Big\},
\end{split}
\end{equation}
where the infimum is taken over all feasible potential functions $\Phi\colon [0, 1]\times \Omega\rightarrow \mathbb{R}$ and parameter paths $\theta\colon [0,1]\rightarrow \mathbb{R}^d$.
\end{proposition}
\begin{proof}
Denote \begin{equation*}
\frac{d}{dt}g(\theta_t,z)=v(t,g(\theta,z)), \quad \textrm{with}\quad v(t, g(\theta_t,z))=D_p H(g(\theta_t,z), \nabla\Phi(t, g(\theta_t,z))).
\end{equation*}
We show that the probability density transition equation of $g(\theta_t,z)$ satisfies the {\em constrained continuity equation}
\begin{equation}\label{1}
 \frac{\partial}{\partial t}\rho(\theta_t, x)+\nabla\cdot(\rho(\theta_t, x)v(t,x))=0,
\end{equation}
and 
\begin{equation}\label{2}
\mathbb{E}_{z\sim\mu}L(g(\theta_t,z), \frac{d}{dt}g(\theta_t,z))=\int_{\Omega}L(x, v(t,x))\rho(\theta_t,x)dx.
\end{equation}

On the one hand, consider $f\in C^{\infty}_c(\Omega)$, then
\begin{equation}\label{form1}
\begin{split}
\frac{d}{dt}\mathbb{E}_{z\sim\mu} f(g(\theta_t,z))=&\frac{d}{dt}\int_{Z} f(g(\theta_t,z))\mu(z)dz\\
=&\frac{d}{dt}\int_\Omega f(x)\rho(\theta_t,x)dx\\
=&\int_\Omega f(x) \frac{\partial}{\partial t}\rho(\theta_t,x)dx,
\end{split}
\end{equation}
where the second equality holds from the push forward relation \eqref{a}. 

On the other hand, consider
\begin{equation}\label{form2}
\begin{split}
\frac{d}{dt}\mathbb{E}_{z\sim\mu} f(g(\theta_t,z))=&\lim_{\Delta t\rightarrow 0}\mathbb{E}_{z\sim\mu}\frac{f(g(\theta_{t+\Delta t}, z)-f(g(\theta_t,z))}{\Delta t}\\
=&\lim_{\Delta t\rightarrow 0}\int_Z\frac{f(g(\theta_{t+\Delta t}, z))-f(g(\theta_t, z))}{\Delta t}\mu(z)dz\\
=&\int_Z \nabla f(g(\theta_t,z)) \frac{d}{dt}g(\theta_t,z)\mu(z)dz\\
=&\int_Z \nabla f(g(\theta_t, z)) v(t, g(\theta_t, z))\mu(z)dz\\
=&\int_\Omega \nabla f(x) v(t,x)\rho(\theta_t, x)dx\\
=&-\int_\Omega f(x) \nabla\cdot(v(t,x)\rho(\theta_t,x))dx,
\end{split}
\end{equation}
where $\nabla$, $\nabla\cdot$ are gradient and divergence operators w.r.t. $x\in\Omega$. The second to last equality holds from the push forward relation \eqref{a}, and the last equality holds using the integration by parts w.r.t. $x$. Since $\eqref{form1}= \eqref{form2}$ for any $f\in C_c^{\infty}(\Omega)$, we have proven \eqref{1}. 

In addition, by the definition of the push forward operator \eqref{a}, we have 
\begin{equation*}
\begin{split}
\mathbb{E}_{z\sim\mu} L(g(\theta_t, z), \frac{d}{dt}g(\theta_t,z))=&\int_{Z} L(g(\theta_t,z), v(t, g(\theta_t,z)))\mu(z) dz\\
=&\int_\Omega L(x, v(t,x))\rho(\theta_t,x)dx.
\end{split}
\end{equation*}
Thus we prove \eqref{2}.
\end{proof}
It is interesting to compare variational problems \eqref{BBL} with \eqref{Lag}. We can view $g(\theta_t, z)\in \Omega$ as ``parameterized'' particles, whose density function is constrained in the parameterized probability set $\rho(\Theta)$. Their motions result at the evolution of probability transition densities in $\rho(\Theta)$, satisfying the constrained continuity equation \eqref{3b}. For this reason, we call \eqref{Lag} the Lagrangian formalism of constrained dynamical OT. 

It is also worth noting that each movement of $g(\theta_t, z)$ results a motion in density path $\rho(\theta_t,x)$. The change of density path will identify a potential function $\Phi(t,x)$ depending on $\theta_t$. 

In additon, the cost functional in dynamical OT can involve general potential energies, such as linear potential energy: 
\begin{equation*}
\mathcal{V}(\rho)=\int_{\Omega}V(x)\rho(x)dx,
\end{equation*}
and interaction energy: 
\begin{equation*}
\mathcal{W}(\rho)=\int_{\Omega}\int_{\Omega}w(x,y)\rho(x)\rho(y)dxdy.
\end{equation*}
Here $V(x)$ is a linear potential, and $w(x,y)=w(y,x)$ is a symmetric interaction potential. If $\rho(\theta, x)\in \rho(\Theta)$, then 
\begin{equation*}
\begin{split}
\mathcal{V}(\rho(\theta, \cdot))=&\int_{\Omega}V(x)\rho(\theta, x)dx\\
=&\int_{Z} V(g(\theta,z))\mu(z)dz\\
=&\mathbb{E}_{z\sim\mu}V(g(\theta, z)),
\end{split}
\end{equation*}
and 
\begin{equation*}
\begin{split}
\mathcal{W}(\rho(\theta, \cdot))=&\int_{\Omega}\int_{\Omega} w(x,y) \rho(\theta, x)\rho(\theta, y)dxdy\\
=&\int_{Z}\int_{Z}w(g(\theta, z_1), g(\theta, z_2))\mu(z_1)\mu(z_2)dz_1dz_2\\
=&\mathbb{E}_{(z_1,z_2)\sim \mu\times \mu}w(g(\theta, z_1), g(\theta, z_2)),
\end{split}
\end{equation*}
where each second equality in the above two formulas hold because of the constrained mapping relation \eqref{a} and $\mu\circ\mu$ represents an independent joint density function supported on $Z\times Z$ with marginals $\mu(z_1)$, $\mu(z_2)$. Similarly in proposition \ref{thm}, we have 
\begin{equation*}\label{MFG}
\begin{split}
c(\theta_0,\theta_1)=&\inf_{v_t,~\theta(0)=\theta_0,~\theta(1)=\theta_1}~\Big\{\int_0^1\Big[\int_\Omega L(x, v(t,x))\rho(\theta_t,x)dx-\mathcal{V}(\rho(\theta_t, \cdot))-\mathcal{W}(\rho(\theta_t,\cdot))\Big]dt\colon\\
&\hspace{4cm} \frac{\partial}{\partial t}\rho(\theta_t,x)+\nabla\cdot (\rho(\theta_t,x)v(t,x))=0\Big\} \\
=&\inf_{\Phi_t,~\theta_t~\theta(0)=\theta_0,~\theta(1)=\theta_1}~\Big\{\int_0^1\Big[\int_\Omega L(x, D_p H(x,\nabla\Phi(t,x)))\rho(\theta_t,x)dx-\mathcal{V}(\rho(\theta_t, \cdot))-\mathcal{W}(\rho(\theta_t,\cdot))\Big]dt\colon\\
&\hspace{4cm} \frac{\partial}{\partial t}\rho(\theta_t,x)+\nabla\cdot (\rho(\theta_t,x)v(t,x))=0\Big\} \\
=&\inf_{\Phi_t,~\theta_t,~\theta(0)=\theta_0,~\theta(1)=\theta_1}\Big\{\int_0^1~[\mathbb{E}_{z\sim\mu}L(g(\theta_t,z), \frac{d}{dt}g(\theta_t, z))-\mathbb{E}_{z\sim\mu}V(g(\theta_t, z))\\
&\hspace{3.5cm}-\mathbb{E}_{(z_1,z_2)\sim \mu\times \mu}w(g(\theta_t, z_1), g(\theta_t, z_2))]dt\colon\\
&\hspace{4cm}\frac{d}{dt}g(\theta_t, z)=D_p H(g(\theta_t,z), \nabla_x\Phi(t, g(\theta_t,z)))\Big\}.
\end{split}
\end{equation*}

We next demonstrate an example of constrained dynamical OT problems.
\begin{example}[Constrained $L^2$-Wasserstein distance]\label{example2}
Let $L(x,v)=\|v\|^2$ and denote $d_{W_2}(\theta_0,\theta_1)=c(\theta_0,\theta_1)^{\frac{1}{2}}$, then
\begin{equation*}
\begin{split}
d_{W_2}(\theta_0, \theta_1)^2=\inf \Big\{&\int_0^1\mathbb{E}_{z\sim\mu}\|\frac{d}{dt}g(\theta_t, z))\|^2dt\colon\frac{d}{dt}g(\theta_t, z)=\nabla\Phi(t, g(\theta_t,z)),~\theta(0)=\theta_0,~\theta(1)=\theta_1\Big\}.
\end{split}
\end{equation*}
Observe that \eqref{Lag} forms a geometric action energy function in parameter space $\Theta$, in which the metric tensor can be extracted explicitly. In other words, denote $G(\theta)\in\mathbb{R}^{d\times d}$ by
\begin{equation*}
\dot\theta^{\ts}G(\theta)\dot\theta=\dot\theta^{\ts}\mathbb{E}_\mu(\nabla_\theta g(\theta,z)\nabla_\theta g(\theta,z)^{\ts})\dot\theta,
\end{equation*}
with the constraint \begin{equation*}
(\dot\theta, \nabla_\theta g(\theta ,z))= \nabla_x\Phi(g(\theta,z)).
\end{equation*}
Here $\nabla_\theta g(\theta,z)\in \mathbb{R}^{d\times n}$, $\Phi$ is a potential function satisfying $$-\nabla\cdot(\rho(\theta,x)\nabla\Phi(x))=(\nabla_\theta\rho(\theta,x), \dot\theta),$$ and $G(\theta)=\mathbb{E}_{z\sim\mu}(\nabla_\theta g(\theta,z) \nabla_\theta g(\theta, z)^{\ts})\in \mathbb{R}^{d\times d}$ is a semi-positive definite matrix. 
\end{example}
\bibliographystyle{abbrv}
\bibliography{LF}
\end{document}